\documentclass[a4paper, 12pt]{amsart}

\usepackage{amscd}

\newtheorem{theorem}{Theorem}[section]
\newtheorem{proposition}[theorem]{Proposition}
\newtheorem{lemma}[theorem]{Lemma}

\theoremstyle{definition}

\newtheorem{definition}[theorem]{Definition}
\newtheorem{remark}[theorem]{Remark}
\newtheorem{example}[theorem]{Example}

\newcommand{\C}{\mathbb{C}}

\newcommand{\Z}{\mathcal{Z}}
\newcommand{\J}{\mathcal{J}}

\newcommand{\D}{\bar{\partial}}
\newcommand{\dbar}{\bar{\partial}}
\newcommand{\Res}{\operatorname{Res}}

\begin{document}

\title[]{A local Grothendieck duality theorem for Cohen-Macaulay ideals}

\author{Johannes Lundqvist}
\thanks{}

\address{Department of Mathematics\newline Stockholm University\newline SE-106 91 Stockholm\newline Sweden}

\email{johannes@math.su.se}

\begin{abstract}
We give a new proof of a recent result due to Mats Andersson and Elizabeth Wulcan,
generalizing the local Grothendieck duality theorem. It can also be seen
as a generalization of a previous result by Mikael Passare. Our method does
not require the use of the Hironaka desingularization theorem and it provides a semi-explicit 
realization of the residue that is annihilated by functions from the given
ideal.
\end{abstract}

\maketitle


\section{Introduction}
Let $\mathcal{O}_0$ be the ring of germs of holomorphic functions at $0\in\mathbb{C}^n$ and let $\Omega_0^n$ denote the germs of holomorphic $(n,0)$-forms.
The ring $\mathcal{O}_0$ is Noetherian and hence all ideals $\mathcal{J}\subset\mathcal{O}_0$ will be finitely generated. Assume first that $\J$ is generated by 
$n$ functions $f=(f_1,\ldots,f_n)$ and that their common zero set consists of one single point, the origin.
Then the Grothendieck residue, $\Res_f$, is defined as 
\begin{equation}\label{OrginalRes}
\Res_f(\xi) = \left( \frac{1}{2\pi i}\right)^n\int_{|f_i(z)|=\epsilon}\frac{\xi(z)}{f_1(z)\cdots f_n(z)},\quad \xi\in\Omega_0^n,
\end{equation}
and is independent of $\epsilon$.
Observe that we can multiply $\Res_f$ with a holomorphic germ $\varphi$ by letting $\varphi\Res_f(\xi)=\Res_f(\varphi\xi)$.
There is a well known theorem, see for example \cite{GH}, saying that $\J$ is equal to the annihilator ideal of $\Res_f$, i.e.,
\begin{equation}
\varphi\Res_f(\xi)=0,\quad \forall\xi\in\Omega_0^n, \quad\text{ iff}\quad \varphi\in\J.
\end{equation}
We will refer to that theorem as the local Grothendieck duality theorem.

There is a cohomological interpretation of the Grothendieck residue.
Let $\Omega$ be an open neighborhood of $0$ such that $f_j$, $j=1,\ldots,n$, and $\xi$ are defined there.
Let $D_j=\{z;f_j(z)=0\}$ and $U_j=\Omega\setminus D_j$.
Then $\xi /f_1\ldots f_n$ can be considered as an $(n-1)$-cochain for the sheaf of holomorphic $(n,0)$-forms and the covering $\{ U_j \}_{j=1,\ldots,n}$ of $\Omega\setminus\{0\}$. Since there are no $(n-1)$-coboundaries, $\xi /f_1\ldots f_n$ defines a \v Cech cohomology class and by the Dolbeault theorem, \cite{GH}, we get a Dolbeault cohomology class $\omega^{\xi}$ of bidegree $(n,n-1)$.
The Grothendieck residue can now be rewritten as integration of $\omega^{\xi}$ over the boundary of a small neighborhood, $D$, of the origin,
\begin{equation}\label{GrotRes}
\Res_f(\xi)=\int_{\partial D} \omega^{\xi}.
\end{equation}
A proof of this can be seen in \cite{GH} where one can also see a proof of the fact that the class $\omega$ can be 
represented by the explicit form
\begin{equation*}
\omega^{\xi} = \left(\frac{1}{2\pi i}\right)^n n!\frac{\sum (-1)^{i-1}\bar{f_i} d\bar{f_1}\wedge\ldots\wedge \widehat{d\bar{f_i}}\wedge\ldots\wedge d\bar{f_n}\wedge\xi}{\left(   |f_1|^2 +\ldots +|f_n|^2   \right)^n },
\end{equation*}
where $\widehat{d\bar{f_i}}$ means that $d\bar{f_i}$ is omitted.

Assume now that the ideal $\J$ is generated by $p$ functions $f_1,\ldots,f_p$ and that we do not have any restrictions on the common zero set $\mathcal{Z}$. With the use of Hironaka's desingularization theorem one can define a residue current
\begin{equation}\label{CH}
\dbar\frac{1}{f_p}\wedge\ldots\wedge\dbar\frac{1}{f_1}\cdot\xi = \lim_{\delta\to 0}\int_{|f_i(z)|=\epsilon_i(\delta)}\frac{\xi(z)}{f_1(z)\cdots f_p(z)},
\end{equation}
for smooth test forms $\xi$.
In order for the limit to exist it has to be taken over a so called admissible path meaning that $\epsilon_i(\delta)$ tends faster to zero than any power of $\epsilon_{i+1}(\delta)$. The current \eqref{CH} is called the Coleff-Herrera product and was defined in \cite{CH}. In the special case of $p=n$ and $\mathcal{Z}=\{0\}$ we get the Grothendieck residue if we restrict the Coleff-Herrera product to the holomorphic germs.

In \cite{DickensteinSessa} and \cite{Passare} Dickenstein-Sessa and Passare independently 
proved that the Coleff-Herrera product satisfies the duality theorem, i.e, that the annihilator ideal of \eqref{CH} is equal to $\mathcal{J}$, in the case when $\J$ defines a complete intersection. That is, the case when the codimension of $\J$ is equal to $p$.
Passare also defines a cohomological residue satisfying the duality theorem in that case generalizing the Grothendieck duality theorem to complete intersections.
In \cite{AnderssonWulcan} Andersson and Wulcan construct a residue current that satisfies a duality theorem for arbitrary ideals and coincides with the Coleff-Herrera product if the ideal defines a complete intersection. They also show that the residue can be expressed cohomologically in the Cohen-Macaulay case.

In this paper we find a new proof of the result of Andersson and Wulcan in the Cohen-Macaulay case avoiding using Hironaka desingularization used in \cite{AnderssonWulcan}.
The residue is similar to (\ref{GrotRes}) and
is obtained from a double complex defined from a free resolution of $\J$.
In the special case of a complete intersection it coincides with the cohomological residue in \cite{Passare}.


\section{Set up and statement}

Remember that a local Noetherian ring $R$ is called Cohen-Macaulay if the maximal length of a regular sequence in $R$ is equal to the dimension of $R$. 
An ideal $\J\subset R$ is called Cohen-Macaulay if $R/\J$ is Cohen-Macaulay. 
All ideals in $\mathcal{O}_0$ whose variety is zero-dimensional are Cohen-Macaulay.
Also, all ideals in $\mathcal{O}_0$ that define a complete intersection are Cohen-Macaulay but the converse is not true. For example, the ideal $\langle z^2,zw,w^2 \rangle\subset\mathcal{O}_0$ is Cohen-Macaulay (because its variety is zero-dimensional) but do not define a complete intersection.

Assume that the common zero set $\mathcal{Z}$ of $f_1,\ldots,f_m\in\mathcal{O}_0$ has pure codimension $p$ and that $\J=\langle f_1,\ldots,f_m\rangle$ is Cohen-Macaulay.
The fact that $\J$ is Cohen-Macaulay is equivalent (because of the Auslander-Buchsbaum formula \cite{Eisenbud}) 
to the existence of a minimal free resolution of $\mathcal{O}_0 / \J$
\begin{equation}\label{complex1}
0\to\mathcal{O}^{\oplus r_p}_0 \overset{f^p}{\to}\mathcal{O}^{\oplus r_{p-1}}_0\overset{f^{p-1}}{\to}
\ldots\overset{f^2}{\to}\mathcal{O}^{\oplus r_1}_0\overset{f^1}{\to}\mathcal{O}_0\to\mathcal{O}_0 / \J\to 0
\end{equation}
having length $p$.
Here $f^1$ can be represented as the row-matrix where the the i:th column is $f_i$ and $f^k, k>1$, are matrices with holomorphic functions as entries.
Oka's lemma, \cite{GH}, implies that there exists a small neighborhood $\Omega$ around $0$ such that the complex
\begin{equation}\label{firstkompl}
0\to\mathcal{O}^{\oplus r_p}_z \overset{f^p}{\to}\mathcal{O}^{\oplus r_{p-1}}_z\overset{f^{p-1}}{\to}
\ldots\overset{f^2}{\to}\mathcal{O}^{\oplus r_1}_z\overset{f^1}{\to}\mathcal{O}_z\to\mathcal{O}_z / \J\to 0
\end{equation}
is exact for all $z\in\Omega$.

\noindent
If we let $E_j$ be a trivial vector bundle of rank $r_j$ over $\Omega$
we get an induced complex of trivial vector bundles
\begin{equation}\label{vectorkomplex}
0\to E_p\overset{f^p}{\to}E_{p-1} \overset{f^{p-1}}{\to}\ldots
\overset{f^2}{\to}E_1\overset{f^1}{\to}E_0\to 0.
\end{equation}
Note that $\mathcal{O}_z / \J=0$ if $z\in\Omega\setminus \mathcal{Z}$ and that 
the complex (\ref{vectorkomplex}) is pointwise exact there.
Indeed, assume that $k<p$ and that $(z_0,x)\in(\Omega\setminus\mathcal{Z})\times\mathbb{C}^{r_k}$ is a point such that $f^k(z_0)x=0$. 
Note first that there exist a non-zero function $\varphi\in\mathcal{O}_{z_0}^{\oplus r_{k}}$ such that $f^k\varphi=0$ because otherwise $\operatorname{Ker}f^k=\{0\}$ and hence $k=p$ which is a contradiction since we assumed that $k<p$.
Take such a $\varphi$. We know from the exactness of \eqref{firstkompl} that there exists $\psi\in\mathcal{O}_{z_0}^{\oplus r_{k+1}}$ such that $f^{k+1}\psi=\varphi$. By scaling we can assume that $\varphi(z_0)=x$ and by choosing $y=\psi(z_0)$ we get that the point $(z_0,y)\in(\Omega\setminus\mathcal{Z})\times\mathbb{C}^{r_{k+1}}$ is mapped to $(z_0,x)$.

The exactness of \eqref{vectorkomplex} and a simple induction over $k$ shows that $f^k$ has constant rank in $\Omega\setminus\mathcal{Z}$ and thus $\operatorname{Ker}f^k$ is a sub-bundle of $E_k$. 
Since $f^{k+1}$ is a pointwise surjection to the sub-bundle $\operatorname{Ker}f^k$ we get that the corresponding complex of smooth sections is exact.
We have just proved the following proposition.

\begin{proposition}\label{CompProp}
Let $\mathcal{E}_{0,q}(\Omega,E_k)$ denote the set of smooth\\ $(0,q)$-sections of $E_k$ over $\Omega$. With $f^k,E_k,\Omega$ and $\mathcal{Z}$ as above, the complex
\begin{align*}
0\to \mathcal{E}_{0,q}(\Omega\setminus\mathcal{Z},E_p)\overset{f^p}{\to}\mathcal{E}_{0,q}(\Omega\setminus\mathcal{Z},E_{p-1})&\overset{f^{p-1}}{\to}\ldots\\
&\ldots\overset{f^1}{\to}\mathcal{E}_{0,q}(\Omega\setminus\mathcal{Z},E_0)\to 0
\end{align*}
is exact for all $q$.
\end{proposition}

We are now ready to define the complex that will give us the cohomology classes we need in order to state the main theorem.
The operators $f^j$ and $\bar{\partial}$ define the double complex

\begin{equation}\label{DoubleComplex}
\begin{CD}
@. \vdots    @. \vdots  @.\\
@.       @A{\bar{\partial}}AA      @A{\bar{\partial}}AA		@.\\
\ldots @>{f^{k+1}}>> \mathcal{E}_{0,q+1}(\Omega\setminus\mathcal{Z},E_k)   @>{f^k}>>   \mathcal{E}_{0,q+1}(\Omega\setminus\mathcal{Z},E_{k-1}) @>{f^{k-1}}>>\ldots\\
@.       @A{\bar{\partial}}AA      @A{\bar{\partial}}AA		@.\\
\ldots @>{f^{k+1}}>> \mathcal{E}_{0,q}(\Omega\setminus\mathcal{Z},E_k)   @>{f^k}>>   \mathcal{E}_{0,q}(\Omega\setminus\mathcal{Z},E_{k-1}) @>{f^{k-1}}>>\ldots\\
@.       @A{\bar{\partial}}AA      @A{\bar{\partial}}AA		@.\\
@. \vdots    @. \vdots  @.
\end{CD}.
\end{equation}
\noindent
Let $\mathcal{L}$ be the total complex of \eqref{DoubleComplex}, i.e.,
\begin{equation}\label{LKomplex}
\mathcal{L}:\ldots \overset{\nabla_f}{\to}\mathcal{L}^r(\Omega\setminus\Z)\overset{\nabla_f}{\to}\mathcal{L}^{r+1}(\Omega\setminus\Z)\overset{\nabla_f}{\to}\ldots.
\end{equation}
where 
\begin{equation*}
\mathcal{L}^r(\Omega\setminus\Z)= \bigoplus_k\mathcal{E}_{0,k+r}(\Omega\setminus\Z,E_k)
\end{equation*}
and
\begin{equation*}
\nabla_f: \mathcal{L}^r(\Omega\setminus\Z)\to\mathcal{L}^{r+1}(\Omega\setminus\Z) \text{ is defined as }
\nabla_f=f-\bar{\partial}.
\end{equation*}
Here $f$ should be interpreted as $(-1)^qf^k$ on $\mathcal{E}_{0,q}(\Omega\setminus\mathcal{Z},E_k)$.
We know from Proposition \ref{CompProp} that the double complex \eqref{DoubleComplex} has exact rows and by a standard spectral sequence argument
we see that the total complex $\mathcal{L}$ is exact.

Now, let $\varphi\in\mathcal{O}_0$. Then we can view $\varphi$ as an element in $\mathcal{L}^{0}(\Omega\setminus\mathcal{Z})$ for some $\Omega$ such that the complex $\mathcal{L}$ is exact. Moreover, $\nabla_f\varphi=0$ so there exists an element $v$ in $\mathcal{L}^{-1}(\Omega\setminus\mathcal{Z})$ such that $\nabla_f v = \varphi$. 
If we write $v=v_1+\ldots + v_p$ where $v_k\in\mathcal{E}_{0,k-1}(\Omega\setminus\mathcal{Z},E_k)$ we see that $f v_1 = \varphi$ and $f v_j = \bar{\partial}v_{j-1}$ for $j=2,\ldots p$. Especially we get that $\bar{\partial}v_p = 0$. Now, if $v,w\in\mathcal{L}^{-1}(\Omega\setminus\mathcal{Z})$ are such that $\nabla_f v = \nabla_f w = \varphi$,
then there exists an element $u\in\mathcal{L}^{-2}(\Omega\setminus{Z})$ such that $\nabla_f u = v - w$ and thus $\bar{\partial}u_p=v_p-w_p$. This means that $v_p$ (a vector of $r_p$ $\bar{\partial}$-closed smooth $(0,p-1)$-forms) is a representative of a Dolbeault cohomology class $\omega^{\varphi}$ of bidegree $(0,p-1)$ depending only on $\varphi$ and $f$, i.e., we have a map
\begin{equation*}
\mathcal{O}_0\ni\varphi\mapsto \omega^{\varphi}=[\alpha_1,\alpha_2,\ldots,\alpha_{r_p}]\in \left( H_{\bar{\partial}}^{(0,p-1)}(\Omega\setminus\mathcal{Z})\right)^{\oplus r_p}.
\end{equation*}
Note that these cohomology classes form a $\mathcal{O}_0$-module and that $\varphi\omega^1=\omega^{\varphi}$ for $\varphi\in\mathcal{O}_0$.

Let X be a subset of $\mathbb{C}^n$ and denote $\omega^1$ by $\omega$. By $\mathcal{D}^{p,q}(X)$ we mean the space of all $(p,q)$-forms that have compact support on X.
\begin{definition}\label{def}
The residue
\begin{equation*}
\Res_f:\{ \xi\in\mathcal{D}_0^{n,n-p}(\Omega);\D\xi = 0 \text{ close to } \mathcal{Z} \} \to \C
\end{equation*}
is given by
\begin{equation}
\Res_f(\xi) = \int\D\xi\wedge\omega.
\end{equation}
\end{definition}
\noindent
The fact that $\Res_f$ is well-defined, i.e., does not depend on the choice of representant of $\omega$, is a direct consequence of Stokes' theorem.
We define multiplication with a holomorphic germ $\varphi$ analogous to the case of the Grothendieck residue, i.e., $\varphi\Res_f(\xi)=\Res_f(\varphi\xi)$
and we thus get
\begin{equation*}
\varphi\Res_f(\xi) = \Res_f(\varphi\xi)=\int\D(\varphi\xi)\wedge\omega=\int(\D(\xi))\wedge\varphi\omega=\int\D\xi\wedge\omega^{\varphi}.
\end{equation*}

\begin{remark}\label{remark}\rm{
If $\mathcal{Z}$ consists of one single point we can rewrite $\Res_f$ in a different way.
Let $\xi\in\mathcal{D}_0^{n,0}$ such that $\D\xi = 0$ close to $0$. Then there exists a compact set $D\subset\Omega$, with 0 in the interior, such that $\D\xi = 0$ on $D$. If $\tilde{\xi}$ is a holomorphic $(n,0)$-form that satisfy $\tilde{\xi}=\xi$ on $D$ we get}
\begin{align*}
\int \bar{\partial}\xi\wedge\omega=
\int_{\mathbb{C}^n \setminus D}\bar{\partial}\xi\wedge\omega=
-\int_{\partial D}\xi\wedge\omega=
-\int_{\partial D}\tilde{\xi}\wedge\omega.
\end{align*}
{\rm We will use this in Example \ref{example} below.}
\end{remark}

The following theorem is the main result in this paper.
\begin{theorem}\label{Main}
Assume that $f_1,\ldots,f_m\in\mathcal{O}_0$ and that the ideal $\J$ generated by the $f_i$:s is Cohen-Macaulay. Then the following are equivalent:
\begin{enumerate}
 \item[(i)] $\varphi\in\J$
 \item[(ii)] $\omega^{\varphi} = 0$
 \item[(iii)] $\varphi\Res_f = 0$
\end{enumerate}
\end{theorem}
We postpone the proof to the next section.

\begin{remark}
The operator $\nabla_f$ was first introduced by Mats Andersson in \cite{Andersson} and was later used in several papers to define residue currents that coincide with the Coleff-Herrera product in the case of complete intersection.
The advantage of using $\nabla_f$ to define the residue $\Res_f$ is that much of the work in the proof of Theorem \ref{Main} is hidden in the construction of the cohomology classes $\omega^{\varphi}$.
\end{remark}


\begin{example}\label{example} {\rm

Consider the case when $\J=\langle f_1,\ldots,f_p \rangle$ defines a complete intersection. It is well known, \cite{BrunsHerzog},
that the
Koszul complex with coefficients in $\mathcal{O}_0$, i.e., the complex 
\begin{equation*}
0 \to\mathcal{O}_0\otimes\Lambda^p E\overset{\delta_f}{\to} 
\ldots \overset{\delta_f}{\to}\mathcal{O}_0\otimes\Lambda^2 E\overset{\delta_f}{\to}\mathcal{O}_0\otimes E \overset{\delta_f}{\to} \mathbb{C}
\to 0,
\end{equation*}
where $E$ is a complex vector space of dimension $p$ with a basis $e_1,\ldots, e_p$ and where $\delta_f$ is defined as 
\begin{gather*}
\delta_f:\mathcal{O}_0\otimes\Lambda^{k}E\rightarrow\mathcal{O}_0\otimes\Lambda^{k-1} E,\\
\delta_f (\psi\otimes e_{l_1}\wedge\ldots\wedge e_{l_k}) = \psi\sum_{j=1}^n(-1)^{j+1} f_j\otimes e_{l_1}\wedge\ldots\wedge\widehat{e_{l_j}}\wedge\ldots\wedge e_{l_k}
\end{gather*}
is a minimal resolution of $\mathcal{O}_0/\J$. This means that the resolution (\ref{complex1}) is isomorphic to the Koszul 
complex since all minimal resolutions are isomorphic.
In this case $\mathcal{L}^r(\Omega\setminus\Z)$ and $\nabla_f$ in the total complex (\ref{LKomplex}) become 
\begin{equation*}
\mathcal{L}^r(\Omega\setminus\Z)= \bigoplus_k\mathcal{E}_{0,k+r}(\Omega\setminus\Z,E_k)\quad\text{and}\quad \nabla_f=\delta_f-\dbar
\end{equation*}
where $E_k$ is the trivial bundle $\Omega\times\Lambda^k E$.
We define the operator
\begin{equation*}
\cap:\mathcal{E}_{0,r}(\Omega\setminus\Z, E_k) \times \mathcal{E}_{0,s}(\Omega\setminus\Z,E_l)\rightarrow
\mathcal{E}_{0,r+s}(\Omega\setminus\Z,E_{k+l})
\end{equation*}
by letting
\begin{equation*}
d z_I\otimes e_J\cap d z_K \otimes e_L = d z_I\wedge d z_K \otimes e_J\wedge e_L.
\end{equation*}
Let us try to calculate the cohomology class $\omega$ in this case.
Let 
\begin{equation*}
\sigma = \frac{\sum_{j=1}^p\bar{f}_j\otimes e_j}{|f|^2}\quad\text{and}\quad v=\sigma\cap(1+\dbar\sigma+(\dbar\sigma)^{\cap 2}+\ldots + (\dbar\sigma)^{\cap (p-1)}).
\end{equation*}
Then $v\in\mathcal{L}^{-1}(\Omega\setminus\mathcal{Z})$ and since $\nabla_f\sigma = 1$ and $(\dbar\sigma)^{\cap p}=0$ we get that $\nabla_f v = 1$.
This means that a representative for the class $\omega$ is given by $v_p=\sigma\cap(\dbar\sigma)^{\cap (p-1)}$,
and by using that $(\sum_{j=1}^p\bar{f}_j\otimes e_j) ^{\cap 2}=0$ we get that
\begin{equation*}
v_p=\frac{\sum \bar{f}_j\otimes e_j \cap (\sum\dbar \bar{f}_j\otimes e_j)^{\cap (p-1)}}{|f|^{2p}}.
\end{equation*}
Now, $\dbar\bar{f}_j\otimes e_j\cap\dbar\bar{f}_k\otimes e_k =\dbar\bar{f}_k\otimes e_k\cap\dbar\bar{f}_j\otimes e_j$ for all $j,k=1,\ldots,p$ and since 
$\dbar\bar{f}_k=d\bar{f}_k$ we get 
\begin{equation*}
v_p = p!\frac{\sum (-1)^{j-1}\bar{f_j} d\bar{f_1}\wedge\ldots\wedge \widehat{d\bar{f_j}}\wedge\ldots\wedge d\bar{f_p}\otimes e_1\wedge\ldots\wedge e_p}{\left(   |f_1|^2 +\ldots +|f_n|^2   \right)^p}.
\end{equation*}
This shows that in the case of a complete intersection the residue coincide with the cohomological residue in \cite{Passare} and
together with Remark \ref{remark} this shows that $\Res_f$ indeed is a generalization of the Grothendieck residue (\ref{GrotRes}).
}
\end{example}


\section{The proof of Theorem \ref{Main}}
We will need a result that describes when we can solve the $\bar{\partial}$-equation in our situation and also a variant of Hartogs' phenomenon. To prove those results we use an integral representation of smooth $(p,q)$-forms called Koppelman's formula.

Let $\Delta = \{(z,z);z\in\mathbb{C}^n\}\subset \mathbb{C}^n\times\mathbb{C}^n$ and 
\begin{equation*}
b(z)=\frac{\partial |z|^2}{2\pi i|z|^2}.
\end{equation*}
A form $s(\zeta,z)$ in $\Omega\times\Omega$ on the form $s(\zeta,z)=\sum s_j(\zeta_j,z_j)d(\zeta_j - z_j)$ that satisfies $2\pi i\sum s_j(\zeta_j,z_j)(\zeta_j-z_j)=1$ outside the diagonal $\Delta$ and $s(\zeta,z)=b(\zeta-z)$ in a neighborhood of $\Delta$ is called an admissible form (in the sense of Andersson) \cite{Andersson}. For an admissable form $s$ one can prove that $K=s\wedge(\bar{\partial}s)^{n-1}$ is $\bar{\partial}$-closed outside 
$\Delta$. By $K_{p,q}$ we mean the component of $K$ that has bidegree $(p,q)$ in $z$ and $(n-p,n-q-1)$ in $\zeta$.
If $f$ is a smooth $(p,q)$-form then for $z\in D$ it has the representation
\begin{equation*}
f(z)=\bar{\partial}_z\!\!\!\int\limits_{\zeta\in D}\!\! K_{p,q-1}(\zeta,z)\wedge f(\zeta) + \!\int\limits_{\zeta\in D} \!\!K_{p,q}(\zeta,z)\wedge \bar{\partial}f(\zeta) +\!\int\limits_{\zeta\in \partial D}\!\!\!\! K_{p,q}(\zeta,z)\wedge f(\zeta).
\end{equation*}
This representation is referred to as Koppelman's formula. If we want to solve the equation $\bar{\partial}u=f$, where $f$ is $\bar{\partial}$-closed 
in some region $D$, 
Koppelman's formula tells us that it is possible if we can make the boundary integral disappear.

\begin{remark}{ \rm
Koppelman's formula is often stated so that the form $s(\zeta,z)$ is equal to $b(\zeta-z)$, see for example \cite{Demailly}.
The formula above follows from the ordinary Koppelman's formula. One way to see this is to first fix $z_0\in D$ and then write $f=\chi f + (1-\chi)f$ where $\chi$ is a cutoff function with suppport in a small neighborhood $U$ of $z_0$ such that $s(\zeta,z)=b(\zeta-z)$ in $U$. The formula now follows from the ordinary Koppelman's formula because of the $\dbar$-closeness of $K$.
}
\end{remark}

\begin{lemma}\label{lemma2}
Write $\mathbb{C}^n=\mathbb{C}^{n-k}\times\mathbb{C}^k$ and $z=(z^{\prime},z^{\prime\prime})$, $\zeta=(\zeta^{\prime},\zeta^{\prime\prime})$.
Assume that $f$ is a $\bar{\partial}$-closed smooth $(0,q)$-form in $\mathbb{B}=\mathbb{B}^{\prime}\times\mathbb{B}^{\prime\prime}$, where
$\mathbb{B}^{\prime}$ and $\mathbb{B}^{\prime\prime}$ are the Euclidean $(n-k)$ and $k$-balls,
and that $f$ has compact support in the $z^{\prime \prime}$ direction.
Then there exists a solution to $\bar{\partial}u = f$ in a possibly smaller set with compact support in the $z^{\prime \prime}$ direction if $q<k$.
If $q = k$ such a solution exists if and only if
\begin{equation}\label{propvill}
\int\xi\wedge f = 0 
\end{equation}
for all $\bar{\partial}$-closed $(n,n-k)$-forms $\xi$ with compact support in the $z^{\prime}$ direction.
\end{lemma}

\begin{proof}
The "only if" part of the statement when $q=k$ is clear because if
there is a solution $u$ to $\bar{\partial}u=f$ with compact support in the $z^{\prime\prime}$ direction then
\begin{equation*}
\int\xi\wedge f = \int\xi\wedge\bar{\partial}u = \int\bar{\partial}(\xi\wedge u) = 0 
\end{equation*}
for all $\bar{\partial}$-closed $(n,n-k)$-forms $\xi$ with compact support in the $z^{\prime}$ direction
by Stokes' theorem, since $\xi\wedge u$ has compact support.

Let $\chi^{\prime}$ be a cutoff function in $\mathbb{B}^{\prime}$ that is equal to $1$ in a neighborhood of $\overline{r\mathbb{B}^{\prime}}$, where $r<1$ and let
$\chi^{\prime\prime}$ be a cutoff function in $\mathbb{B}^{\prime\prime}$ that is equal to $1$ in a neighborhood of $\overline{r\mathbb{B}^{\prime\prime}}$.
Set 
\begin{align*}
s(\zeta,z) = &\chi^{\prime} (\zeta^{\prime}) 
\bigg[ 
\chi^{\prime\prime} (z^{\prime\prime})b(\zeta-z) + (1-\chi^{\prime\prime} (z^{\prime\prime}))
\frac{\bar{z}^{\prime\prime}\cdot d(\zeta-z)}{2\pi i(|z^{\prime\prime}|^2 -\zeta^{\prime\prime}\cdot\bar{z}^{\prime\prime})}
\bigg]+\\
+ &(1 - \chi^{\prime} (\zeta^{\prime}))
\bigg[
\frac{\bar{\zeta}^{\prime}\cdot d(\zeta-z)}{2\pi i(|\zeta^{\prime}|^2 - z^{\prime}\cdot\bar{\zeta}^{\prime})}
\bigg].
\end{align*}
Then $s(\zeta,z)$ is admissible for $|z^{\prime}|\leq r$ and for $|\zeta^{\prime\prime}|\leq r$.

Note that we can extend $s$ to an admissible form for $|\zeta^{\prime\prime}|<1$ simply 
by considering $\chi s + ( 1 - \chi )b$ where $\chi$ is a cutoff function in $r\mathbb{B}$. Since we can assume that $\chi$ is $1$ in 
$\operatorname{supp}f$ this extension will be of no interest since $K\wedge f = 0$ outside the suppport of $f$. This means that for our $s$, Koppelman's formula 
will work for all $z^{\prime\prime}$.

If $|\zeta^{\prime}|$ is close to $1$ we get 
\begin{equation*}
s(\zeta,z)=\frac{\bar{\zeta}^{\prime}\cdot d(\zeta-z)}{2\pi i(|\zeta^{\prime}|^2 - z^{\prime}\cdot\bar{\zeta}^{\prime})},
\end{equation*}
which is holomorphic in $z$. Therefore the boundary integral in Koppelman's formula vanishes if $q \geq 0$
since $f$ has compact support in the $\zeta^{\prime\prime}$ direction.
Thus $u(z)=\int K_{0,q-1}\wedge f$ is a solution to $\bar{\partial}u = f$. It remains to show that the solution has 
compact support in the $z^{\prime\prime}$ direction. 
Let $|z^{\prime\prime}|$ be close to $1$. Then 
\begin{align*}
s(\zeta,z) &=\chi^{\prime} (\zeta^{\prime})
\bigg[ 
\frac{\bar{z}^{\prime\prime}\cdot d(\zeta-z)}{2\pi i(|z^{\prime\prime}|^2 -\zeta^{\prime\prime}\cdot\bar{z}^{\prime\prime})}
\bigg]
+ (1 - \chi^{\prime} (\zeta^{\prime}))
\bigg[
\frac{\bar{\zeta}^{\prime}\cdot d(\zeta-z)}{2\pi i(|\zeta^{\prime}|^2 - z^{\prime}\cdot\bar{\zeta}^{\prime})}
\bigg]\\
&=: s_1(\zeta,z) + s_2(\zeta,z).
\end{align*}
We see that $\bar{\partial}_{z}s_2 = 0$ and that both $s_1$ and $s_2$ are $\dbar_{\zeta^{\prime\prime}}$-closed. 
This means that $K_{0,q-1} = 0$ if $q<k$ because of degree reasons since then $n-q>n-k$ and $K_{0,q-1}$ have bidegree $(n,n-q)$ in $\zeta$.
In the case $q=k$ we will show that $K_{0,q-1}$ is $\bar{\partial}_{\zeta}$-closed and has compact support in the $\zeta^{\prime}$-direction. This will actually 
end the proof since then we can use \eqref{propvill} with $\xi = K_{0,k-1}$.
\newline\noindent
Assume $q=k$. Then 
$K_{0,k-1}$ have bidegree $(n,n-k)$ in $\zeta$
and thus $K_{0,k-1}$ 
is $\bar{\partial}_{\zeta}$-closed since we get too many $\zeta^{\prime}$ differentials.
Assume now that $|\zeta^{\prime}|$ and $|z^{\prime\prime}|$ are close to $1$. 
Then
\begin{equation*}
s(\zeta,z)=\frac{\bar{\zeta}^{\prime}\cdot d(\zeta-z)}{2\pi i(|\zeta^{\prime}|^2 - z^{\prime}\cdot\bar{\zeta}^{\prime})},
\end{equation*}
and since it 
do not contain $\zeta^{\prime\prime},z^{\prime\prime},\bar{\zeta}^{\prime\prime}$ or $,\bar{z}^{\prime\prime}$
we may regard it as an admissible form on $\mathbb{B}^{\prime}\times\mathbb{B}^{\prime}$.
In particular, this means that $K=s\wedge (\bar{\partial} s)^{n-k-1}$ is $\bar{\partial}$-closed outside of $\Delta$ which means that $K_{0,k-1}=0$.
\end{proof}

\begin{proposition}[Variant of Hartogs' phenomenon]\label{hartog}
Let $\Omega=\Omega^{\prime}\times\Omega^{\prime\prime}$, where $\Omega^{\prime\prime}$ has dimension $k>1$, be an open set in 
$\mathbb{C}^n$ and let $K=\Omega^\prime\times\overline{r\mathbb{B}}$ for some $r<1$ such that $\overline{r\mathbb{B}}\subset\Omega^{\prime\prime}$. 
If $q<k-1$ then for each smooth $\bar{\partial}$-closed $(0,q)$-form $\nu$ in $(\Omega\setminus K)$ there exists a $\dbar$-closed $(0,q)$-form $\hat{\nu}$ in $\Omega$ such that $\hat{\nu}=\nu$ in 
$\Omega\setminus \hat{K}$ where $\hat{K}$ is a slightly bigger set than $K$. If $q = 0$ we have $\hat{K}=K$.
If $q=k-1$ the above statement is true if 
\begin{equation*}
\int \bar{\partial}\xi\wedge \nu = 0 
\end{equation*}
for all $(n,n-k)$-forms $\xi$ with compact support that are $\bar{\partial}$-closed in a neighborhood of $K$.
\end{proposition}

\begin{proof}
Let $\chi$ be a cutoff function in $\Omega$ that is identically $1$ in a neighborhood of $K$ and let $g:=(-\bar{\partial}\chi)\wedge\nu$. Then $g$ is $\bar{\partial}$-closed 
in $\Omega$ and
\begin{equation*}
\int\xi\wedge g = -\int\xi\wedge\bar{\partial}\chi\wedge \nu =  \pm\int\bar{\partial}(\xi\wedge\chi)\wedge \nu = 0,
\end{equation*}
for $\bar{\partial}$-closed $(n,n-q-1)$-forms $\xi$ with compact support in the $z^{\prime}$ direction. 
This means that we can use Lemma \ref{lemma2} with $g$ as $f$ and thus there exists a solution to $\bar{\partial}u=g$, with compact support in the $z^{\prime\prime}$ 
direction, in a possibly smaller set. 
Set $\hat{\nu}= (1-\chi)\nu - u$. Then $\dbar\hat{\nu}=0$ and $\hat{\nu}=\nu$ close to the boundary where $|z^{\prime\prime}|=1$.
If $q=0$ the uniqueness theorem for analytic functions imply that $\hat{\nu}=\nu$ in $\Omega\setminus K$.
\end{proof}

\begin{proof}{(Proof of Theorem \ref{Main})}\\
$(i)\Rightarrow(ii)$:
Assume that $\varphi \in \mathcal{J}$.
Let $\Omega$ be an open neighborhood of the origin such that $\mathcal{L}$ is exact for $\Omega\setminus\mathcal{Z}$ and such that there exist 
functions $\psi_j\in \mathcal{O}(\Omega)$, such that 
\begin{equation*}
\varphi = \sum \psi_j f_j.
\end{equation*}
Let $\{ e_j \}$ be a global frame of $E_1$ such that $f^1(1\otimes e_j)=f_j$
and let $v=v_1+\ldots+v_p\in\mathcal{L}^{-1}$ be defined by letting $v_1 = \sum\psi_j\otimes e_j$ and $v_2=v_3=\ldots=v_p=0$. Then $\nabla_f v_1 = \varphi$ and $\omega^{\varphi}=0$ and we are done.\\
$(ii)\Rightarrow(iii)$:
Trivial.\\
$(iii)\Rightarrow (i)$:
Assume that $\varphi\in\mathcal{O}_0$ and that $\varphi\Res_f=0$.
Let again $\Omega$ be such that $\mathcal{L}$ is exact for $\Omega\setminus\mathcal{Z}$ and
let $v=v_1 + v_2 + \ldots + v_p\in\mathcal{L}^{-1}(\Omega\setminus\mathcal{Z})$ be a solution to $\nabla v = \varphi$.
Because of general properties of complex analytic sets 
we may assume that $\Omega$ is the set $\mathbb{B}^{\prime}\times \mathbb{B}^{\prime\prime}$, where $\mathbb{B}^{\prime}\subset\mathbb{C}^{n-p}$ and $\mathbb{B}^{\prime\prime}\subset\mathbb{C}^p$ are the Euclidean balls,
and that $\mathcal{Z}$ do not touch the boundary of $r\mathbb{B}^{\prime\prime}$ for some $r<1$, \cite{Chirka}.
According to Proposition \ref{hartog} we can extend $v_p$ to a $\dbar$-closed form $\hat{v}_p$ in $\Omega$ since $v_p$ fulfills the requirement by the
assumption that $\varphi\Res_f=0$.
We can now solve the equation $\bar{\partial}u_p=\hat{v}_p$ and since $\hat{v}_p=v_p$ close to the boundary where $|z^{\prime\prime}|=1$ there exists a solution in say $\Omega\setminus K$ where $K$ is a set of the same type as in Proposition \ref{hartog}.
Now, in $\Omega\setminus K$ we get that
\begin{align*}
\bar{\partial}(v_{p-1}+f^p u_p) = f^p v_p - f^p \hat{v}_p = 0.
\end{align*}
This means that there exists a solution to $\bar{\partial} u_{p-1} = v_{p-1} + f^p u_p$ in $\Omega\setminus K$ and we note that 
\begin{align*}
\bar{\partial} (v_{p-2}+ f^{p-1} u_{p-1}) &= f^{p-1} v_{p-1} + \bar{\partial} f^{p-1} u_{p-1}= f^{p-1} v_{p-1}- f^{p-1} v_{p-1}\\ &= 0.
\end{align*}
If we repeat the argument above we eventually end up with 
\begin{equation*}
\bar{\partial}(v_1 + f^2 u_2) = 0
\end{equation*}
in a smaller set of the same type, call it $\mathcal{U}$.
Now,
\begin{equation*}
f^1\psi = f^1 v_1 + f^1 f^2 u_2 = f^1 v_1 = \varphi
\end{equation*}
in $\mathcal{U}$ and Proposition \ref{hartog} in the case where $q=0$ completes the proof.
\end{proof}

{ \bf Acknowledgement: } The author would like to thank professor Mats Andersson for suggesting the problem and for valuable comments on the early versions of the paper.
The author would also like to thank the referee for his/her comments and suggestions which helped improve the paper.


\end{document}